\documentclass[12pt]{amsart}

\usepackage{mathrsfs}
\usepackage{amsthm}
\newtheorem{theorem}{Theorem}
\newtheorem{lemma}{Lemma}
\newtheorem{corollary}{Corollary}
\newtheorem{definition}{Definition}

\title{ON THE DISCRETE FUNCTIONAL $L_p$ MINKOWSKI PROBLEM}
\author{Tuo Wang}
\address {Institut f\"{u}r Diskrete Mathematik und Geometrie\\ TU Wien\\
Wiedner Hauptstrasse 8-10, 1040 Wien\\ Austria}
\email{tuowang734@gmail.com}
\begin{document}
\maketitle

\begin{abstract}
The discrete functional $L_p$ Minkowski problem is posed and solved. As a consequence, the general affine P\'{o}lya-Szeg\"{o} principle
and the general affine Sobolev inequalities are established.
\end{abstract}

{\noindent 2000 AMS subject classification: 46B20 (46E35, 52A21, 52B45)}\\
{\noindent Keywords: affine P\'{o}lya-Szeg\"{o} principle, functional Minkowski problem, convex symmetrization.}

\section{Introduction}
The even functional Minkowski problem was first introduced by Lutwak, Yang and Zhang in \cite{20} to answer the following question:

Given a function $f\in W^{1,1}(\mathbb R^n)$, where $W^{1,1}(\mathbb R^n)$ is the usual $L^1$ Sobolev space on $\mathbb R^n$, which norm $\|\cdot\|_o$ on $\mathbb R^n$ minimizes the quantity $\int_{\mathbb R^n}\|\nabla f\|_odx,$ if the unit ball of the dual norm of $\|\cdot\|_o$ has the same volume as the Euclidean unit ball in $\mathbb R^n$?

The solution to this problem turns out to be unique and is called the optimal Sobolev norm.

To answer the question, Lutwak, Yang and Zhang introduced the following even functional Minkowski problem on $W^{1,1}(\mathbb R^n):$

Given $f\in W^{1,1}(\mathbb R^n)$, find an $o$-symmetric convex body $\langle f\rangle_e$ such that
\begin{equation}\label{1}
\int_{S^{n-1}}\Psi(u)dS(\langle f\rangle_e,u)=\int_{\mathbb R^n}\Psi(\nabla f(x))dx
\end{equation}
for all continuous functions $\Psi$	on $\mathbb R^n$ that are even and positively homogeneous of degree $1$. Here $S(\langle f\rangle_e,\cdot)$ is the Alexandrov-Fenchel-Jessen surface area measure of $\langle f\rangle_e$ and $o$-symmetric stands for origin-symmetric.

The optimal Sobolev body $\langle\cdot\rangle_e$ defined through equation (\ref{1}) provides the solution of the optimal Sobolev norm problem. 
The even functional Minkowski problem is shown to play an important role in the theory of affine Sobolev inequalities on $\mathbb R^n$ (see [13][20][27][29]). Using valuations on Sobolev spaces, Ludwig obtained a natural description of the operator $f\mapsto\langle f\rangle_e$ (see \cite{13}). By extending the even functional Minkowski problem to $BV(\mathbb R^n),$ the space of functions of bounded variations, the author was able to extend the Zhang-Sobolev inequality to $BV(\mathbb R^n)$ and got a characterization of the equality cases (see \cite{27}). By combining the even functional Minkowski problem with the convex symmetrization technique introduced in \cite{1}, the author got a Brothers-Ziemer type theorem for the affine P\'{o}lya-Szeg\"{o} principle, which clarified the equality condition of the affine P\'{o}lya-Szeg\"{o} principle (see \cite{29}). Based on these facts, it is natural to ask the following question:

Is there always a hidden convex body that plays an important role in each affine Sobolev type inequality?

While for the symmetric optimal $L_p$ affine Sobolev inequality, the answer was obtained by the Lutwak, Yang and Zhang \cite{20}. The aim of this paper is to define and solve the general functional $L_p$ Minkowski problems on Sobolev spaces. By making use of the solutions of the general functional $L_p$ Minkowski problems and the general $L_p$ affine isoperimetric inequalities \cite{7}, we arrive at the general affine P\'{o}lya-Szeg\"{o} principle and general affine Sobolev type inequalities, which include the results in [8][9][22] as special cases. This shows that the crucial convex body behind the discrete version general affine Sobolev is the solution of the discrete functional Minkowski problem.

To to specific, we study the discrete versions of functional $L_p$ Minkowski Problems. Instead of $W^{1,1}(\mathbb R^n)$, we consider the functional $L_p$ Minkowski problem on $L^{1,0}(\mathbb R^n)$, a dense subspace of $W^{1,p}(\mathbb R^n)$.
 \begin{definition}The space of piecewise affine functions on $\mathbb R^n$, where a function $f: \mathbb R^n\rightarrow \mathbb R$ is called piecewise affine, if $f$ is continuous and there are finitely many $n$-dimensional simplices $T_1,...,T_m\subset \mathbb R^n$ with pairwise disjoint interiors such that the restriction of $f$ to each $T_i$ is affine and $0$ outside $T1\cup\cdots\cup T_m.$\end{definition}

 In particular, we get the following result:
 \begin{theorem} Given a nontrivial function $f\in L^{1,0}(\mathbb R^n)$, there exists a
unique polytope $\langle f\rangle_p\in\mathcal P_0^n$ such that
\begin{equation}\label{4}
n\int_{\mathbb R^n}\Psi^p(-\nabla f(x))dx=\int_{S^{n-1}}\Psi(u)^pdS_p(\langle f\rangle_p,u),
\end{equation}
for every continuous function $\Psi:\mathbb R^n\rightarrow[0,\infty)$ that is homogeneous of
degree $1$.
\end{theorem}

The new functional $L_p$ Minkowski problem enjoys some natural affine invariance property and valuation property. By studying
this new functional Minkowski problem on $L^{1,0}(\mathbb R^n)$, we are led to the general affine P\'{o}lya-Szeg\"{o} principle and general affine Sobolev inequalities.

We remark that, although the discovery of the discrete $L^p$ Minkowski problem on $W^{1,p}(\mathbb R^n)$ is the main issue of this piece of work, a very important tool that we make use in this piece of work is the Haberl-Schuster version of the general $L_p$ affine isoperimetric inequality in \cite{7}.



\section{Background material}
Background on the $L_p$ Brunn-Minkowski theory. For quick reference we recall in this section some background material from the $L_p$ Brunn-Minkowski theory of convex bodies. This theory has its origin in the work of Firey from the 1960's and has expanded rapidly over the last two decades since the work of Lutwak \cite{17} (see, e.g.,[2, 3, 4, 7, 9, 10, 11, 12, 14, 16, 17, 18, 19, 20, 21, 22, 23, 24, 27, 28]).

A convex body is a compact convex subset of $\mathbb R^n$ with non-empty interior. We write $\mathcal K^n$ for the set of convex bodies in $\mathbb R^n$ endowed with the Hausdorff metric and we denote by $\mathcal K^n_e$ the set of origin symmetric convex bodies, by $\mathcal K^n_0$ the set of convex bodies containing origin in the interior and by $\mathcal P^n_0$ the set of polytopes containing origin in the interior. Each non-empty compact convex set $K$ is uniquely determined by its support function $h(K,\cdot)$, defined by
$$h(K,x)=\max\{x\cdot y: y\in K\},~~x\in\mathbb R^n,$$
where $x\cdot y$ denotes the Euclidean inner product of $x$ and $y$ in $\mathbb R^n$. Note that $h(K,\cdot)$ is positively homogeneous of degree $1$ and sub-additive. Conversely, every function with these two properties is the support function of a unique compact convex set.

For later use, we include the following lemma from \cite{26}:

\begin{lemma}
Let $K_i$ and $K$ be convex bodies in $\mathbb R^n$. If $h(K_i,\cdot)\rightarrow h(K,\cdot)$ pointwise, then $h(K_i,\cdot)\rightarrow h(K,\cdot)$ uniformly on $S^{n-1}$.
\end{lemma}

If $K\in \mathcal K^n_0$ , then the polar body $K^\circ$ of $K$ is defined by $$K^\circ=\{x\in \mathbb R^n: x\cdot y\leq 1~~~ for ~all~ y\in K\}.$$

From the polar formula for volume it follows that the $n$-dimensional Lebesgue measure $|K^\circ|$ of the polar body $K^\circ$ can be computed by
$$|K^\circ|=\frac{1}{n}\int_{S^{n-1}}h(K,u)^{-n}du,$$ where integration is with respect to spherical Lebesgue measure.

For real $p\geq 1$ and $\alpha,\beta>0$, the $L_p$ Minkowski combination of $K,L\in\mathcal K_0^n$ is the convex body $\alpha\cdot K+_p\beta\cdot L$ defined by $$h^p(\alpha\cdot K+_p\beta\cdot L,u)=h^p(K,u)+h^p(L,u),$$ and the $L_p$ mixed volume of $K,L$ is defined by

 $$V_p(K,L)=\frac{p}{n}\lim_{\epsilon\rightarrow0^+}\frac{|K+_p\epsilon\cdot L|-|K|}{\epsilon}.$$

Clearly, the diagonal form of $V_p$ reduces to ordinary volume, i.e., for $K\in\mathcal K_0^n,$ $$V_p(K,K)=|K|.$$

It was also shown in \cite{17} that for all convex bodies $K,L\in\mathcal K_0^n,$ $$V_p(K,L)=\frac{1}{n}\int_{S^{n-1}}h(L,u)^pdS_p(K,u),$$ where $dS_p(K,u)=h(K,u)^{1-p}dS(K,u)$. Recall that for a Borel set $\omega\subset S^{n-1}$, $S(K,\omega)$ is the $(n-1)$-dimensional Hausdorff measure $H^{n-1}$ of the set of all boundary points of $K$ for which there exists a normal vector of $K$ belonging to $\omega.$

A fundamental inequality that we need is the $L_p$ Minkowski inequality \cite{17}.

\begin{theorem} [\cite{17}] If $1<p<\infty$ and $K,L\in\mathcal K_0^n,$ then
\begin{equation}\label{2}
V_p(K,L)\geq|K|^{1-p/n}|L|^{p/n}.
\end{equation}
Equality holds if and only if $L=tK$ for some $t>0.$
\end{theorem}

Projection bodies have become a central notion within the Brunn-Minkowski theory. They arise naturally in a number of different areas such as functional analysis, stochastic geometry and geometric tomography. The fundamental affine isoperimetric inequality which connects the volume of a convex body with that of its polar projection body is the Petty projection inequality \cite{25}. This inequality turned out to be stronger than the classical isoperimetric inequality and its functional form is Zhang's discovery, namely the affine Zhang-Sobolev inequality \cite{30}.

The $L_p$ projection body $\Pi_pK$ of $K\in\mathcal K_0^n,$ was introduced in \cite{18} as the convex body such that
$$h(\Pi_pK,u)^p=\int_{S^{n-1}}|u\cdot v|^pdS_p(K,v), ~u\in S^{n-1}.$$

The $L_p$ analog of Petty's projection inequality plays a key role in the $L_p$ Brunn-Minkowski theory. It was first proved by Lutwak, Yang, and Zhang \cite{21} (see also Campi and Gronchi \cite{3} for an independent approach): If $K\in\mathcal K_0^n$, then
$$|K|^{n/p-1}|\Pi_pK|\leq|B_2|^{n/p-1}|\Pi_pB_2|,$$
where $B_2$ is the Euclidean unit ball and equality is attained if and only if $K$ is an ellipsoid centered at the origin.

For a finite Borel measure $\mu$ on $S^{n-1}$, we define a continuous function $C_p^{+}(\mu)$ on $S^{n-1}$, the asymmetric $L_p$ cosine transform of $\mu$, by
$$(C_P^+\mu)(u)=\int_{S^{n-1}}(u\cdot v)^p_+d\mu(v), u\in S^{n-1},$$
where $(u\cdot v)_+=\max\{u\cdot v,0\}$. The asymmetric $L_p$ projection body $\Pi^p_+(K)$ of $K\in\mathcal K_0^n,$ first considered in \cite{12}, is the convex body defined by $$h(\Pi^p_+(K),\cdot)^p=C_p^+(S_p(K,\cdot)).$$

For $p>1$, by using valuations, Ludwig \cite{12} established the $L_p$ analogue of her classification of the projection operator: She showed that the convex bodies
$$(1-\lambda)\cdot\Pi_p^+(K)+_p\lambda\cdot\Pi_p^-(K), K\in\mathcal K_0^n$$
where $0\leq\lambda\leq1$ and $\Pi_p^-(K)=\Pi_p^+(-K)$, constitute all natural $L_p$ extensions of projection bodies.

Haberl and Schuster established the following general $L_p$ Petty projection inequalities in \cite{7}.
\begin{theorem} [\cite{7}] Let $K\in\mathcal K_0^n$ and $p>1$. If $\Phi_{\lambda,p}K$ is the convex body defined by
\begin{equation}\label{3}
\Phi_{\lambda,p}(K)=(1-\lambda)\cdot\Pi_p^+(K)+_p\lambda\cdot\Pi^-_p(K),
\end{equation}
where $0\leq\lambda\leq1$, then
$$|K|^{n/p-1}|\Phi_{\lambda,p}K|\leq|B_2|^{n/p-1}|\Phi_{\lambda,p}B_2|,$$
with equality if and only if $K$ is an ellipsoid centered at the origin.
\end{theorem}

The classical Minkowski problem has a counterpart in the $L_p$ Brunn-Minkowski theory. The $L_p$ Minkowski problem asks the following: Given a Borel measure $\mu$ on $S^{n-1}$, does there exist a convex body $K$ such that $\mu=S_p(K,\cdot)$?

If the data $\mu$ is an even measure, Lutwak \cite{17} gave an affirmative answer to this problem when $p\neq n$. Later on, Lutwak, Yang and Zhang \cite{24} introduced the volume-normalized even $L_p$ Minkowski problem, for which the case $p=n$ can be handled as well. What we need in this work is the following discrete version of the $L_p$ Minkowski problem.

\begin{theorem} [\cite{10}] Let vectors $u_1,...,u_m\in S^{n-1}$ that are not contained in a closed hemisphere and real numbers $\alpha_1,...,\alpha_m>0$ be given. Then, for any $p>1$ with $p\neq n$, there exists a unique polytope $P\in\mathcal P_0^n$ such
that $$\sum_{j=1}^{m}\alpha_j\delta_{u_j}=S_p(P,\cdot).$$
\end{theorem}

For treating the case $p=n$, we need the normalized version of the above theorem.

\begin{theorem} [\cite{10}] Let vectors $u_1,...,u_m\in S^{n-1}$ that are not contained in a closed hemisphere and real numbers $\alpha_1,...,\alpha_m>0$ be given. Then there exists a unique polytope $P\in\mathcal P_0^n$ such that $$\sum_{j=1}^{m}\alpha_j\delta_{u_j}=\frac{S_n(P,\cdot)}{|P|}.$$
\end{theorem}

\section{The discrete functional $L_p$ Minkowski problem}
We assume $1<p<\infty$ and $p\neq n$ throughout this section.

\subsection{The discrete case.}
\begin{theorem} Given a nontrivial function $f\in L^{1,0}(\mathbb R^n)$, there exists a
unique polytope $\langle f\rangle_p\in\mathcal P_0^n$ such that
\begin{equation}\label{4}
n\int_{\mathbb R^n}\Psi^p(-\nabla f(x))dx=\int_{S^{n-1}}\Psi(u)^pdS_p(\langle f\rangle_p,u),
\end{equation}
for every continuous function $\Psi:\mathbb R^n\rightarrow[0,\infty)$ that is homogeneous of
degree $1$.
\end{theorem}
\begin{proof}
Let $\Sigma=\{x:\nabla f(x)=0\}.$ Then the map $$\Upsilon\mapsto n\int_{\mathbb R^n\backslash \Sigma}\Upsilon(-\frac{\nabla f(x)}{|\nabla f(x)|})|\nabla f(x)|^pdx$$
defines a nonnegative bounded linear functional on the space of continuous functions on $S^{n-1}$. It follows from the Riesz representation
theorem that there exists a unique Borel measure $S_p(f,\cdot)$ on $S^{n-1}$ such that
$$n\int_{\mathbb R^n\backslash\Sigma}\Upsilon(-\frac{\nabla f(x)}{|\nabla f(x)|})|\nabla f(x)|^pdx=\int_{S^{n-1}}\Upsilon(u)dS_p(f,u),$$
for each continuous function $\Upsilon: S^{n-1}\rightarrow\mathbb R.$

Choose $\Upsilon=\Psi^p$ and radially extend $\Psi$ to $\mathbb R^n$ by requiring $\Psi$ to be of homogeneity degree $1$. Then we have
\begin{eqnarray*}
& &n\int_{\mathbb R^n}\Psi^p(-\nabla f(x))dx\\
&=&n\int_{\mathbb R^n\backslash \Sigma}\Psi^p(-\nabla f(x))dx\\
&=&n\int_{\mathbb R^n\backslash \Sigma}\Psi^p(-\frac{\nabla f(x)}{|\nabla f(x)|})|\nabla f(x)|^pdx\\
&=&\int_{S^{n-1}}\Psi^p(u)dS_p(f,u).
\end{eqnarray*}

We claim that $S_p(f,u)$ is a finite combination of Dirac measures on the sphere for any $f\in L^{1,0}(\mathbb R^n)$.

Notice from the definition, we have
$$S_p(f(x),\cdot)=S_p(f(x+x_0),\cdot),$$
for any $x_0\in\mathbb R^n.$

Next we consider a triangulation of the support of $f$, i.e. on each
simplex $M_i$, $f$ is affine, where $i=1,...,m.$

Without loss of generality, we assume
$$f(x)=x\cdot u_i+c_i$$
for $x\in M_i$. Then we have
\begin{eqnarray*}
& &n\int_{\mathbb R^n}\Psi^p(-\nabla f(x))dx\\
&=&n\sum_{i=1}^{m}\int_{M_i}\Psi^p(-\nabla f(x))dx\\
&=&n\sum_{i=1}^{m}|M_i|\Psi^p(-u_i)dx.
\end{eqnarray*}
Therefore we have $$S_p(\langle f\rangle_p,\cdot)=n\sum_{i=1}^{m}|M_i|\delta_{-u_i},$$
where $\delta_{-u_i}$ is the Dirac measure concentrated on $-u_i.$

We remark that the fact that $S_p(f,u)$ does not concentrate on any great sub-sphere on $S^{n-1}$ follows basically from the same argument of Lemma 4.1 in \cite{20}. In fact, by tracing back to the definition, one gets $$S_p(f,u)(S^{n-2})\neq 0,$$ where $S^{n-2}$ is any big sphere of $S^{n-1}$.

Therefore, for any $f\in L^{1,0}(\mathbb R^n)$, $S_p(f,u)$ is a linear combination of the Dirac measures on the sphere which do not concentrate on any great subsphere. By using Theorem 4, we prove the existence of the polytope $\langle f\rangle_p\in\mathcal P_0^n.$

The uniqueness is a consequence of the $L_p$ Minkowski inequality (Theorem 2).
\end{proof}

For $P\in\mathcal P_0^n$, define the piecewise affine function $l_P$ by requiring that $l_P(0)=1,$ that $l_P(x)=0$ for $x\not\in P$, and that $l_P$ is affine on each simplex with apex at the origin and base equal to a facet of $P$. Define $\mathcal P^{1,p}\subset L^{1,0}(\mathbb R^n)$ as the set of all $l_P$ for $P\in P_0^n$. Note that for $\phi\in GL(n),$ $$l_{\phi P}=l_P\circ\phi^{-1}.$$

We remark that multiples and translates of $l_P\in P^{1,p}(\mathbb R^n)$ correspond to linear elements within the theory of finite elements.

\begin{lemma}
For $P\in\mathcal P_0^n,$ $$\langle l_P\rangle_p=P.$$
\end{lemma}
\begin{proof}
Let $P$ have facets $F_1,...,F_m.$ For the facet $F_i$, let $u_i$ be its unit outer normal vector and $T_i$ the convex hull of $F_i$ and the origin. Since for $x\in T_i,$
$$l_P(x)=-\frac{u_i}{h(P,u_i)}\cdot x+1$$
and
$$\nabla l_P(x)=-\frac{u_i}{h(P,u_i)},$$
we have that
\begin{eqnarray*}
n\int_{\mathbb R^n}\Psi^p(-\nabla l_P(x))dx &=& n\sum_{i=1}^m\int_{T_i}\Psi^p(-\nabla l_P(x))dx\\
&=& n\sum_{i=1}^m\int_{T_i}\Psi^p(\frac{u_i}{h(P,u_i)})dx\\
&=& n\sum_{i=1}^m\Psi^p(u_i)\frac{V_n(T_i)}{h(P,u_i)^p}\\
&=& \sum_{i=1}^m\Psi^p(u_i)V_{n-1}(F_i)h(P,u_i)^{1-p}\\
&=& \int_{S^{n-1}}\Psi^p(u)dS_p(P,u).
\end{eqnarray*}
Since $S_p(P,\cdot)=S_p(\langle 1_P\rangle_p,\cdot),$ we have $P=\langle 1_P\rangle_p$ from Theorem 4.
\end{proof}

\begin{lemma}
If $f\in L^{1,0}(\mathbb R^n)$ and $\phi\in SL(n)$, then
$$\langle f\circ\phi^{-1}\rangle_p=\phi\langle f\rangle_p.$$
\end{lemma}
\begin{proof}
It is well-known (e.g. see \cite{20}) that if
$$V_p(\phi K,L)=V_p(J,\phi^{-1}L)~ for~ all~ convex~ bodies~ L\in\mathcal K_0^n,$$
then $K=J.$

Since by the affine invariance of the $L_p$ mixed volume, we have
$$V_p(\phi\langle f\rangle_p,L)=V_p(\langle f\rangle_p,\phi^{-1}L),$$
and we know, by the change of variable $y=\phi(x)$, that
\begin{eqnarray*}
\int_{\mathbb R^n}h(L,\nabla(f\circ\phi^{-1})(u))^pdy&=&\int_{\mathbb R^n}h(L,\phi^{-t}\nabla f(\phi^{-1}y))^pdy\\
&=& \int_{\mathbb R^n}h(\phi^{-1}L,\nabla f(x))^pdx.
\end{eqnarray*}
Therefore, we have that $V_p(\langle f\circ\phi^{-1}\rangle_p,L)=V_p(\langle f\rangle_p,\phi^{-1}L),$ which implies that $\langle f\circ\phi^{-1}\rangle_p=\phi\langle f\rangle_p.$
\end{proof}
As a quick consequence of Lemma 3, we have the following corollary.
\begin{corollary}
 Given $\phi\in SL(n)$ and $f\in L^{1,0}(\mathbb R^n)$, we have
 $$|\langle f\rangle_p|=|\langle f\circ\phi^{-1}\rangle_p|.$$
\end{corollary}
The operator $\langle\cdot\rangle_p:L^{1,0}(\mathbb R^n)\rightarrow \mathcal P_0^n$ has the  valuation property.

\begin{definition}
 For two polytopes $K,L\in\mathcal P_0^n$, we define the $L_p$-Blaschke sum $K\sharp_pL$ of $K,L$ by requiring
 \begin{equation}\label{5}
 S_p(K\sharp_pL,\cdot)=S_p(K,\cdot)+S_p(L,\cdot).
 \end{equation}
\end{definition}
Remark: By Theorem 4, we know there is a polytope $K\sharp_pL$ uniquely defined through equation (\ref{5}).

For $f,g\in L^{1,0}(\mathbb R^n)$, the function $f\vee g$ denotes the pointwise maximum and the function $f\wedge g$ the pointwise minimum of $f$ and $g$.

\begin{theorem}
The operator $\langle\cdot\rangle_p:(L^{1,0}(\mathbb R^n),\wedge,\vee)\rightarrow(\mathcal P_0^n,\sharp_p)$ is an $L_p$-Blaschke valuation.
\end{theorem}
\begin{proof}
 What we need to show is that
\begin{eqnarray*}
& &\int_{S^{n-1}}\Psi(u)^pdS_p(\langle f\rangle_p,u)+\int_{S^{n-1}}\Psi(u)^pdS_p(\langle g\rangle_p,u)\\
&=& \int_{S^{n-1}}\Psi(u)^pdS_p(\langle f\wedge g\rangle_p,u)+\int_{S^{n-1}}\Psi(u)^pdS_p(\langle f\vee g\rangle_p,u),
\end{eqnarray*}
for every continuous $\Psi:S^{n-1}\rightarrow\mathbb R$ and $f,g\in L^{1,0}(\mathbb R^n).$

For this we triangulate the support of $f$ and $g$ so that $f(x)=x\cdot u_i+c_i$ on $P_i$ for $i=1,...,m$ and $g(x)=x\cdot v_i+d_i$ on $Q_i$ for $i=1,...,n$, and without loss of generality, we assume
$$
\{x:f(x)\geq g(x)~ and ~g(x)\neq0\}=\bigcup_{k(i)}P_{k(i)}\cup\bigcup_{l(j)}Q_{l(j)},
$$
where $k(i)\in\{1,...,m\}$ and $l(j)\in\{1,...,n\}.$ Then we notice
\begin{eqnarray*}
& &\int_{S^{n-1}}\Psi(u)^pdS_p(\langle f\rangle_p,u)+\int_{S^{n-1}}\Psi(u)^pdS_p(\langle g\rangle_p,u)\\
&=& n(\int_{\mathbb R^n}\Psi(-\nabla f(x))^pdx+\int_{\mathbb R^n}\Psi(-\nabla g(x))^pdx)\\
&=& n(\sum_{i=1}^m\Psi(-u_i)^p|P_i|+\sum_{j=1}^{j=n}\Psi(-v_j)^p|Q_j|).
\end{eqnarray*}
Similarly,
\begin{eqnarray*}
& &\int_{S^{n-1}}\Psi(u)^pdS_p(\langle f\wedge g\rangle_p,u)+\int_{S^{n-1}}\Psi(u)^pdS_p(\langle f\vee g\rangle_p,u)\\
&=& n(\int_{\mathbb R^n}\Psi(-\nabla(f\wedge g)(x))^pdx+\int_{\mathbb R^n}\Psi(-\nabla(f\wedge g)(x))^pdx)\\
&=& n(\sum_{k(i),l(j)}(\Psi(-u_{k(i)})^p|P_{k(i)}|+\Psi(-v_{l(j)})^p|Q_{l(j)}|)\\
& & + \sum_{i\neq k(i)}\Psi(-u_i)^p|P_i|+\sum_{j\neq l(j)}\Psi(-v_j)^p|Q_j|)\\
&=& n(\sum_{i=1}^m\Psi(-u)^p|P_i|+\sum_{j=1}^{n}\Psi(-v_j)^p|Q_j|).
\end{eqnarray*}
This concludes the proof.
\end{proof}

In \cite{13}, Monika Ludwig started the search of interesting valuations on function spaces. She found out that the only reasonable valuation that enjoys some natural affine invariance property is the LYZ operator defined by equation (\ref{1}). Motivated by this, we propose a Ludwig type conjecture: if $Z:(L^{1,0}(\mathbb R^n),\wedge,\vee)\rightarrow(\mathcal P_0^n,\sharp_p)$ is a continuous $L_p$ Blaschke valuation which satisfies
$$Z(f)=Z(f\circ\tau^{-1}), Z(f\circ\phi^{-1})=|\det\phi|^q\phi Z(f)$$
for all $\tau\in\mathbb R^n$, $\phi\in GL(n)$ and some $q\in\mathbb R$, then
$$Z=c\langle\cdot\rangle_p$$
for some constant $c\geq0.$

\subsection {The homothetic functional $L_p$ Minkowski problem.}
Given $f\in W^{1,p}(\mathbb R^n)$, its distribution function $\mu_f:[0,\infty)\rightarrow[0,\infty]$ is defined by
$$\mu_f(t)=|\{x\in\mathbb R^n:|f(x)|>t\}|,$$
where $|\cdot|$ denotes Lebesgue measure on $\mathbb R^n$. The decreasing rearrangement $f^\ast:[0,\infty)\rightarrow[0,\infty]$ of $f$ is defined by $$f^\ast(s)=\inf\{t>0:\mu_f(t)\leq s\}.$$

The symmetric rearrangement $f^\ast:\mathbb R^n\rightarrow[0,\infty]$ of $f$ is the function defined by
$$f^\star(x)=f^\ast(\omega_n|x|^n),$$
where $|\cdot|$ is the standard Euclidean norm.

For a convex body $K\in\mathcal K_0^n$, the convex symmetrization $f^K$ of $f$ with respect to $K$ is defined as follows:
$$f^K(x)=f^\ast(\omega_nh(\tilde{K}^\circ,x)^n),$$
where $h(\tilde{K}^\circ,x)$ is the support function of $\tilde{K}^\circ$, with $\tilde{K}$ being a dilation of $K$ so that $|\tilde{K}|=\omega_n$.

\begin{lemma}\label{forcorollary2}
Given $f^K\in W^{1,p}(\mathbb R^n)$ and $K\in\mathcal K_0^n$, there is a unique convex body $\langle f^K\rangle_p\in\mathcal K_0^n$ such that
$$n\int_{\mathbb R^n}\Psi^p(-\nabla f^K(x))dx=\int_{S^{n-1}}\Psi(u)^pdS_p(\langle f^K\rangle_p,u),$$
for every continous $\Psi$ that is homogeneous of degree $1$. Furthermore
$$\langle f^K\rangle_p=\xi_f\tilde{K},$$
where $\xi_f=(n(n\omega_n)^p\int_0^{\infty}t^{np+n-p-1}((-f^\ast)'(\omega_nt^n))^pdt)^{\frac{1}{n-p}}.$
\end{lemma}
\begin{proof}
We remark that the proof here is similar to that in \cite{29}.

First, we notice that the uniqueness is a quick consequence of the $L_p$
Minkowski inequality (Equation (\ref{1})).

Since $h(\tilde{K}^\circ,\cdot)$ is a Lipschitz function and $h(\tilde{K}^\circ,\cdot)=1$ on $\partial\tilde{K},$ for almost all $x\in\partial\tilde{K}$,
$$\sigma_{\tilde{K}}(x)=\frac{\nabla h(\tilde{K}^\circ,x)}{|\nabla h(\tilde{K}^\circ,x)|},$$
where $\tilde{\sigma_K}(x)$ is the outer unit normal vector of $\tilde{K}$ at the point $x$. And
from the definition of the polar body, we have $$h(\tilde{K},\sigma_K(x))=\frac{1}{|\nabla h(\tilde{K}^\circ,x)|}.$$

Also, it is well known that $f^\ast$ is locally absolutely continuous on $(0,\infty)$. We have, by the co-area formula applied to $h(\tilde{K}^\circ,\cdot)$, that:
\begin{eqnarray*}
& &n\int_{\mathbb R^n}\Psi^p(-\nabla f^K(x))dx\\
&=&n\int_{\mathbb R^n}\Psi^p(-(f^\ast)'(\omega_nh(\tilde{K}^\circ,x)^n)n\omega_nh(\tilde{K}^\circ,x)^{n-1}\nabla h(\tilde{K}^\circ,x))dx\\
&=&n\int_{0}^{\infty}\int_{\partial K}t^{n-1}((-(f^\ast)'(\omega_nt^n))n\omega_nt^{n-1})^p\Psi^p(\frac{\nabla h(\tilde{K}^\circ,x)}{|\nabla h(\tilde{K}^\circ,x)|})|\nabla h(\tilde{K}^\circ,x)|^{p-1}dH^{n-1}(x)dt\\
&=&n(n\omega_n)^p\int_0^\infty t^{np+n-p-1}((-f^\ast)'(\omega_nt^n))^pdt\int_{S^{n-1}}\Psi^p(u)h^{1-p}(\tilde{K},u)dS(\tilde{K},u)\\
&=&n(n\omega_n)^p\int_0^\infty t^{np+n-p-1}((-f^\ast)'(\omega_nt^n))^pdt\int_{S^{n-1}}\Psi^p(u)dS_p(\tilde{K},u)
\end{eqnarray*}
for any $\Psi$ that is homogeneous of degree $1$.

So we get $$S_p(\langle f^K\rangle_p,u)=\xi_fS_p(\tilde{K},u),$$
where $\xi_f=n(n\omega_n)^p\int_0^\infty t^{np+n-p-1}((-f^\ast)'(\omega_nt^n))^pdt.$
This concludes the proof.
\end{proof}

\begin{corollary}
 If $f\in L^{1,0}(\mathbb R^n)$ and $K,L\in \mathcal P_0^n,$ then $\langle f^K\rangle_p$ is a dilate of $K$ and $$|\langle f^K\rangle_p|=|\langle f^L\rangle_p|.$$
\end{corollary}
\begin{proof}
It is an immediate corollary of Lemma 4. Because from Lemma 4, we have $$\langle f^K\rangle_p=\xi_f\tilde{K},$$ and $$\langle f^L\rangle_p=\xi_f\tilde{L}.$$
We know from the definition that $|\tilde{K}|=|\tilde{L}|$, therefore Corollary 2 follows.
\end{proof}

\section{Applications}
\subsection{General $L_p$ Affine Sobolev Inequalities.}
The general $L_p$ Sobolev inequality established in \cite{1} could be stated as follows:
\begin{theorem}[\cite{1}] Let $K\in\mathcal K_0^n$ be normalized such that $|K|=n\omega_n$. Then $$\alpha_{n,p}(\int_{\mathbb R^n}h(K,\nabla f(x))^pdx)^{1/p}\geq(\int_{\mathbb R^n}|f|^{p^\star}dx)^{1/{p^\star}},$$
for $1<p<n$ and $f\in W^{1,p}(\mathbb R^n)$, where $p^\star=\frac{np}{n-p}$. The optimal constants $\alpha_{n,p}$ are given by
\begin{equation}\label{6}
 \alpha_{n,p}=n^{-1/p}(\frac{p-1}{n-p})^{1-1/p}(\frac{\Gamma(n)}{\omega_n\Gamma(\frac{n}{p})\Gamma(n+1-\frac{n}{p})})^{1/n}.
\end{equation}
\end{theorem}
By optimizing this inequality among all norms that have unit balls in $\mathcal P_0^n$, we get the following result:
\begin{theorem}
For $f\in L^{1,0}(\mathbb R^n)$ with $1<p<n$, we have $$(\omega_n)^{1/n}\alpha_{n,p}|\langle f\rangle_p|^{\frac{n-p}{np}}\geq(\int_{\mathbb R^n}|f|^{p^\star}dx)^{1/{p^\star}},$$
where the best constant $\alpha_{n,p}$ is defined in Theorem 8 and $p^\star=\frac{np}{n-p}$.
\end{theorem}
\begin{proof}
Since we have $$\alpha_{n,p}(\int_{\mathbb R^n}h(K,\nabla f(x))^pdx)^{1/p}\geq(\int_{\mathbb R^n}|f|^{p^\star}dx)^{1/{p^\star}},$$
for arbitrary $K\in\mathcal K_0^n$ with fixed volume $n\omega_n$, we have, in particular, $$\alpha_{n,p}(\int_{\mathbb R^n}h(\tilde{\langle f\rangle_p},\nabla f(x))^pdx)^{1/p}\geq(\int_{\mathbb R^n}|f|^{p^\star}dx)^{1/{p^\star}}.$$
Therefore, we have $$\alpha_{n,p}|\langle f\rangle_p|^{\frac{n-p}{np}}\geq(\omega_n)^{-1/n}(\int_{\mathbb R^n}|f|^{p^\star}dx)^{1/{p^\star}}.$$
\end{proof}

\begin{theorem}
Given $0\leq\lambda\leq1$ and a function $f\in W^{1,p}(\mathbb R^n)$ with $1<p<n$, we have $$2^{1/p}\alpha_{n,p}(\int_{S^{n-1}}(\int_{\mathbb R^n}((1-\lambda)(D_vf)^p_++\lambda(D_vf)_-^p)dx)^{-n/p}dv)^{-1/n}\geq(\int_{\mathbb R^n}|f|^{p^\star}dx)^{1/{p^\star}},$$
where the best constant $\alpha_{n,p}$ is defined in Theorem 8 and $p^\star=\frac{np}{n-p}$.
\end{theorem}
\begin{proof}
For $f\in L^{1,0}(\mathbb R^n)$, since we have $$|\langle f\rangle_p|^{\frac{n-p}{np}}\geq(\omega_n)^{-1/n}\alpha_{n,p}^{-1}(\int_{\mathbb R^n}|f|^{p^\star}dx)^{1/{p^\star}},$$
by using Theorem 3 we have $$|\Phi_{\lambda,p}\langle f\rangle_p|^{-1}\geq\alpha|\langle f\rangle_p|^{n/p-1},$$
where $\Phi_{\lambda,p}(\langle f\rangle_p)$ is defined through equation (\ref{3}) as:
$$\Phi_{\lambda,p}(\langle f\rangle_p)=(1-\lambda)\cdot\Pi_p^+(\langle f\rangle_p)+_p\lambda\cdot\Pi_p^-(\langle f\rangle_p),$$
with $0\leq\lambda\leq1$ and $\alpha$ is the best possible constant depending on $p,n$.

Combining the above inequalities we get $$|\Phi_{\lambda,p}\langle f\rangle_p|^{-\frac{1}{n}}\geq\alpha(\int_{\mathbb R^n}|f|^{p^\star}dx)^{1/{p^\star}}.$$
Using the polar formula, we get $$(\int_{S^{n-1}}(\int_{\mathbb R^n}((1-\lambda)(D_vf)^p_++\lambda(D_vf)^p_-)dx)^{-{n/p}}dv)^{-1/n}\geq\alpha(\int_{\mathbb R^n}|f|^{p^\star}dx)^{1/{p^\star}},$$
where $\alpha$ is the best possible constant depending on $n,p.$

For general $f\in W^{1,p}(\mathbb R^n)$, we choose a sequence $f_k\in L^{1,0}(\mathbb R^n),$ such that $f_k\to f$ in $W^{1,p}(\mathbb R^n)$. We know that
$$\int_{\mathbb R^n}((1-\lambda)(D_vf_k)_+^p+\lambda(D_vf_k)^p)dx\to\int_{\mathbb R^n}((1-\lambda)(D_vf)^p_++\lambda(D_vf)^p)dx$$ pointwise for all $v\in S^{n-1}$ as $k\to\infty$.

Since both $(\int_{\mathbb R^n}((1-\lambda)(D_vf_k)_+^p+\lambda(D_vf_k)^p_-)dx)^{1/p}$ and $(\int_{\mathbb R^n}((1-\lambda)(D_vf)^p_++\lambda(D_vf)_-^p)dx)^{1/p}$ are support functions of convex bodies, we
know from Lemma 1 that $$\int_{\mathbb R^n}((1-\lambda)(D_vf_k)_+^p+\lambda(D_vf_k)^p)dx\to\int_{\mathbb R^n}((1-\lambda)(D_vf)^p_++\lambda(D_vf)^p)dx$$ uniformly for all $v\in S^{n-1}$ as $k\to\infty.$

Therefore, using Fatou's lemma we arrive at $$(\int_{S^{n-1}}(\int_{\mathbb R^n}((1-\lambda)(D_vf)^p_++\lambda(D_vf)^p_-)dx)^{-n/p}dv)^{-1/n}\geq\alpha(\int_{\mathbb R^n}|f|^{p^\star})^{1/{p^\star}}$$
for all $f\in W^{1,p}(\mathbb R^n)$, where $\alpha$ is the best possible constant depending on $n,p.$

By calculating the optimal constant $\alpha$, we conclude the proof.
\end{proof}

\subsection{General Affine P\'{o}lya-Szeg\"{o} Principles.}

While the usual P\'{o}lya-Szeg\"{o} principles states that the Dirichlet integral does not increase by symmetric rearrangement. The affine P\'{o}lya-Szeg\"{o} principles states the affine energy, which is affine invariant, defined in \cite{22} does not increase by symmetric rearrangement. Yet the affine
P\'{o}lya-Szeg\"{o} principles is stronger than the usual P\'{o}lya-Szeg\"{o} principles in the sense that the affine P\'{o}lya-Szeg\"{o} principles implies the usual P\'{o}lya-Szeg\"{o} principles by the H\"{o}lder inequality.
\subsubsection {The case $p\neq n$.}

We assume in this subsection that $1<p<\infty$ and $p\neq n$. We will deal with the case $p=n$ in later parts because we need a different normalization factor.

In \cite{1} the authors studied convex symmetrization for general norms. In terms of support functions of convex bodies, the result can be stated as follows:

\begin{theorem} [\cite{1}] For $K\in\mathcal K_0^n,$ $$\int_{\mathbb R^n}h(K,\nabla f(x))^pdx\geq \int_{\mathbb R^n}h(K,\nabla f^K(x))^pdx$$
for all $f\in L^{1,0}(\mathbb R^n).$
\end{theorem}
As an application of the above theorem, we have the following lemma:
\begin{lemma}
Given $f\in L^{1,0}(\mathbb R^n),$ we have $$\langle f\rangle_p=(1+\alpha)\langle f^{\langle f\rangle_p}\rangle_p$$
for some $\alpha\geq0.$
\end{lemma}
\begin{proof}
Since we know
\begin{eqnarray*}
n\int_{\mathbb R^n}h(K,\nabla f(x))^pdx &=& \int_{S^{n-1}}h(K,u)^pdS(\langle f\rangle_p,u)\\&=&nV_p(\langle f\rangle_p,K).
\end{eqnarray*}
Similarly, we have
\begin{eqnarray*}
n\int_{\mathbb R^n}h(K,\nabla f^K(x))^pdx &=& \int_{S^{n-1}}h(K,u)^pdS(\langle f^K\rangle_p,u)\\&=& nV_p(\langle f^K\rangle_p,K).
\end{eqnarray*}
Since from Lemma 4 we know that $\langle f^K\rangle_p$ is a dilation of $K$, we have,
by choosing $K=\langle f\rangle_p$ in the above equation and using Theorem 11
and Theorem 6, that $$V_p(\langle f\rangle_p,\langle f^{\langle f\rangle_p}\rangle_p)\leq V_p(\langle f\rangle_p,\langle f\rangle_p).$$
This implies that $$\langle f\rangle_p=(1+\alpha)\langle f^{\langle f\rangle_p}\rangle_p$$ for some $\alpha\geq0.$
\end{proof}

We are ready to prove the following general affine P\'{o}lya-Szeg\"{o} principle. We remark that the case $\lambda=\frac{1}{2}$ is due to Cianchi, Lutwak, Yang and Zhang \cite{5} and the case $\lambda=0$ is due to Haberl, Schuster and Xiao \cite{9}.
\begin{theorem}
Given $f\in W^{1,p}(\mathbb R^n)$, if we denote
$$\Omega_{\lambda,p}(f)=\alpha_{n,p}(\int_{S^{n-1}}(\int_{\mathbb R^n}((1-\lambda)(D_vf)^p_++\lambda(D_vf)^p_-)dx)^{-n/p}du)^{-1/n},$$
where $0\leq\lambda\leq1$, $\alpha_{n,p}=(n\omega_n)^{1/n}(\frac{n\omega_n\omega_{p-1}}{\omega_{n+p_2}})^{1/p}$ and $\omega_n=\frac{\pi^{n/2}}{\Gamma(1+\frac{n}{2})},$ we have $$\Omega_{\lambda,p}(f)\geq\Omega_{\lambda,p}(f^\star).$$
\end{theorem}
\begin{proof}
If $f\in L^{1,0}(\mathbb R^n)$, by Lemma 5, we have
$$\langle f\rangle_p=(1+\alpha)\langle f^{\langle f\rangle_p}\rangle_p$$
for some $\alpha\geq0.$ Therefore $\Phi_{\lambda,p}(\langle f\rangle_p)\supseteq\Phi_{\lambda,p}(\langle f^{\langle f\rangle_p}\rangle_p),$ where $\Phi_{\lambda,p}(\langle f\rangle_p)$ is defined through Equation (\ref{3}). It is not difficult to see that $$|\Phi_{\lambda,p}(\langle f\rangle_p)|\leq|\Phi_{\lambda,p}(\langle f^{\langle f\rangle_p}\rangle_p)|,$$
therefore we have $$\Omega_{\lambda,p}(f)\geq\Omega_{\lambda,p}(f^{\langle f\rangle_p}).$$
On the other hand, by Lemma 4 and Theorem 3, we have
$$|\Phi_{\lambda,p}(\langle f^{\langle f\rangle_p}\rangle_p)|\leq|\Phi_{\lambda,p}(\langle f^\star\rangle_p)|,$$
which implies $$\Omega_{\lambda,p}(f^{\langle f\rangle_p})\geq\Omega_{\lambda,p}(f^\star).$$
Therefore, we have $$\Omega_{\lambda,p}(f)\geq\Omega_{\lambda,p}(f^\star),$$
for $f\in L^{1,0}(\mathbb R^n).$

For general $f\in W^{1,p}(\mathbb R^n)$, consider an approximation sequence $f_k\in L^{1,0}(\mathbb R^n)$ in the $W^{1,p}(\mathbb R^n)$ norm. We know that $$\Omega_{\lambda,p}(f_k^\star)\leq\Omega_{\lambda,p}(f_k)~~~~for~k\in N.$$
We know that $$\int_{\mathbb R^n}((1-\lambda)(D_vf_k)_+^p+\lambda(D_vf_k)^p_-)dx\to\int_{\mathbb R^n}((1-\lambda)(D_vf)^p_++\lambda(D_vf)^p_-)dx$$
pointwise for all $v\in S^{n-1}$ as $k\to\infty$.

Since both $(\int_{\mathbb R^n}((1-\lambda)(D_vf_k)^p_+\lambda(D_vf_k)^p_-)dx)^{1/p}$ and $(\int_{\mathbb R^n}((1-\lambda)(D_vf)^p_++\lambda(D_vf)^p_-)dx)^{1/p}$ are support functions of convex bodies, we know from Lemma 1 that
$$\int_{\mathbb R^n}((1-\lambda)(D_vf_k)_+^p+\lambda(D_vf_k)^p_-)dx\to\int_{\mathbb R^n}((1-\lambda)(D_vf)^p_++\lambda(D_vf)^p_-)dx$$ uniformly for all $v\in S^{n-1}$ as $k\to\infty.$

Moreover, the function $$v\to(\int_{\mathbb R^n}((1-\lambda)(D_vf)^p_++\lambda(D_vf)^p_-)dx)^{1/p}$$
is strictly positive and continuous on $S^{n-1}$, hence attains a positive minimum on $S^{n-1}$. Consequently, $$(\int_{\mathbb R^n}((1-\lambda)(D_vf_k)^p_++\lambda(D_vf_k)^p)dx)^{-\frac{n}{p}}\to(\int_{\mathbb R^n}((1-\lambda)(D_vf)^p_++\lambda(D_vf)^p_-)dx)^{-\frac{n}{p}}$$
uniformly for $v\in S^{n-1}$ as $k\to \infty$. Therefore $$\lim_{k\to\infty}\Omega_{\lambda,p}(f_k)=\Omega_{\lambda,p}(f).$$

On the other hand, $f^\star\to f^\star$ in $L^p(\mathbb R^n)$, because of the contractivity of the spherically symmetric rearrangement in $L^p(\mathbb R^n).$ Hence we get that $f_k^\star\to f^\star$ weakly in $W^{1,p}(\mathbb R^n).$ Since $\Omega_{\lambda,p}(f_k^\star)=\|\nabla f_k^\star\|_p$ and $\Omega_{\lambda,p}(f^\star)=\|\nabla f^\star\|_p$, and since the $L^p(\mathbb R^n)$ norm of the gradient is lower semi-continuous with respect to weak convergence in $W^{1,p}(\mathbb R^n),$ $$\liminf_{k\to\infty}\Omega_{\lambda,p}(f_k^\star)\geq\Omega_{\lambda,p}(f^\star).$$

Therefore, we have $$\Omega_{\lambda,p}(f^\star)\leq\Omega_{\lambda,p}(f),$$
for all $f\in W^{1,p}(\mathbb R^n).$
\end{proof}

\subsubsection {The case $p=n$.}

For dealing with the case $p=n$, we use the following normalized version of the discrete functional $L_p$ Minkowski problem.
\begin{theorem}
Given a function $f\in L^{1,n}(\mathbb R^n)$, there exists a unique convex polytope $\langle f\rangle_n\in\mathcal P_0^n$ such that
$$\int_{\mathbb R^n}\Psi^n(-\nabla f(x))dx=\frac{1}{|\langle f\rangle_n|}\int_{S^{n-1}}\Psi(u)^ndS_n(\langle f\rangle_n,u),$$ for every continuous function $\Psi:\mathbb R^n\to[0,\infty)$ that is homogeneous of degree $1$.
\end{theorem}

We remark that we omit the proof of Theorem 13, because it is essentially the same as the proof of Theorem 6. We remark that Lemma 6 still holds. We will not give a proof of it. The proof is essentially the same as the proof of Lemma 5, if as we define the normalized $L_n$ optimal Sobolev body $\langle f^K\rangle_n$ as the convex body satisfying
$$\int_{\mathbb R^n}\Psi^n(-\nabla f^K(x))dx=\frac{1}{|\langle f^K\rangle_n|}\int_{S^{n-1}}\Psi(u)^ndS_n(\langle f^K\rangle_n,u),$$
for every continuous function $\Psi:\mathbb R^n\to[0,\infty)$ that is homogeneous of degree $1$.
\begin{lemma}
Given $f\in L^{1,0}(\mathbb R^n)$, we have $$\langle f^{\langle f\rangle_n}\rangle_n=(1+\alpha)\langle f\rangle_n$$
for some $\alpha\geq0.$
\end{lemma}

Given $0\leq\lambda\leq1$, we define the normalized general $L_n$ projection body $\tilde{\Phi}_nK$ of $K$ by: $$h^n(\tilde{\Phi}_{\lambda,n}(K),v)=\frac{1}{|K|}\int_{S^{n-1}}((1-\lambda)(u\cdot v)^n_++\lambda(u\cdot v)^n_-)dS_n(K,u).$$

We need the following normalized version of the general $L_n$ Petty projection inequality.

\begin{lemma} [\cite{20}] If $K$ is a convex body, then
$$\frac{|\tilde{\Phi}_{\lambda,n}K|}{|K|}\leq\frac{|\tilde{\Phi}_{\lambda,n}B_2|}{|B_2|},$$ and equality holds if and only if $K$ is an ellipsoid.
\end{lemma}

With all these tools at hand, we establish the following theorem. We remark again that the case $\lambda=\frac{1}{2}$ is due to Cianchi, Lutwak, Yang and Zhang \cite{5} and the case $\lambda=0$ is due to Haberl, Schuster and Xiao \cite{9}.

\begin{theorem}
Given $f\in W^{1,n}(\mathbb R^n)$, we have
$$\Omega_{\lambda,n}(f)\geq\Omega_{\lambda,n}(f^\star),$$
if we denote $$\Omega_{\lambda,n}(f)=2^{1/n}\alpha_n(\int_{S^{n-1}}(\int_{\mathbb R^n}((1-\lambda)(D_vf)^n_++\lambda(D_vf)^n_-)dx)^{-1}dv)^{-1/n},$$
where $\alpha_n=(n\omega_n)^{1/n}(\frac{n\omega_n\omega_{n-1}}{2\omega_{2n-2}})^{1/n}.$
\end{theorem}
\begin{proof}
By Lemma 6, we have $\langle f^{\langle f\rangle_n}\rangle_n\supseteq\langle f\rangle_n.$ Therefore, we have $$|\tilde{\Phi}^\circ_{\lambda,n}\langle f^{\langle f\rangle_n}\rangle_n|^{-\frac{1}{n}}\leq|\tilde{\Phi}^\circ_{\lambda,n}\langle f\rangle_n|^{-\frac{1}{n}},$$
for $f\in L^{1,n}(\mathbb R^n)$. By Lemma 7 we have
$$|\tilde{\Phi}^\circ_{\lambda,n}\langle f^\star\rangle_n|^{-\frac{1}{n}}\leq|\tilde{\Phi}^\circ_{\lambda,n}\langle f^{\langle f\rangle_n}\rangle_n|^{-\frac{1}{n}}\leq|\tilde{\Phi}^\circ_{\lambda,n}\langle f\rangle_n|^{-\frac{1}{n}},$$
which implies $$\Omega_{\lambda,n}(f)\geq\Omega_{\lambda,n}(f^\star)~~for~all~f\in L^{1,n}(\mathbb R^n).$$

For general $f\in W^{1,n}(\mathbb R^n)$, the approximation argument is the same as in Theorem 12.
\end{proof}

\subsection {General Affine Sobolev Type Inequalities.} For the symmetric case, the affine energy coincides with the Dirichlet integral:
\begin{lemma}
Given $f\in W^{1,p}(\mathbb R^n)$ with $1<p<\infty$ and $0\leq\lambda\leq1$, we
have $$\Omega_{\lambda,p}(f^\star)=\|\nabla f^\star\|_p.$$
\end{lemma}

We omit the proof of Lemma 8 because it is essentially the same as
that in \cite{9}.

With the general affne P\'{o}lya-Szeg\"{o} principle, Theorem 12 and Theorem 14 at hand, it is straightforward to get some affine Sobolev type inequalities which generalize the results in \cite{9}. We will list some (and not all) of them here. Because the proofs are essentially the same as those in \cite{9}, we will omit them.

\begin{corollary} [General affine Moser-Trudinger inequalities] If $f\in W^{1,n}(\mathbb R^n)$ with $0<|sprt f|<\infty$ and $0\leq\lambda\leq1$, then
\begin{equation}\label{7}
\frac{1}{|sprt f|}\int_{sprt f}\exp(\frac{n\omega_n^{1/n}|f(x)|}{\Omega_{\lambda,n}(f)})^{\frac{n}{n-1}}dx\leq m_n.
\end{equation}
\end{corollary}
Here $sprt f$ is the support of $f$ and the constant $n\omega_n^{1/n}$ is optimal, in that equation (7) would fail for any real number $m_n$ if $n\omega_n^{1/n}$ were to be replaced by a larger number. And the best constant $m_n$ is characterized as $$m_n=\sup_g\int_0^\infty\exp(g(t)^{\frac{n}{n-1}}-t)dt,$$
where the supremum ranges over all nondecreasing and locally absolutely continuous functions $g$ on $[0,\infty)$ such that $g(0)=0$ and $\int_0^{\infty}g'(t)^ndt\leq1.$

\begin{corollary} [General affine Morrey-Sobolev inequalities] If $f\in W^{1,p}(\mathbb R^n)$, $p>n,$ such that $|sprt f|<\infty$, then $$\|f\|_{\infty}\leq\alpha_{n,p}|sprt f|^{\frac{p-n}{np}}\Omega_{\lambda,p}(f),$$
where $\alpha_{n,p}=n^{-1/p}\omega_n^{-1/n}(\frac{p-1}{p-n})^{\frac{p-1}{p}}$ is the best constant.
\end{corollary}

Remark: Everything that is done in this paper has a counterpart for the case $p=1$. The extension is similar to what we've done in this paper, as long as one observes $$\int_{\mathbb R^n}\nabla f(x)dx=0$$ for every $f\in W^{1,1}(\mathbb R^n).$

\section{Acknowledgements}
The work of the author was supported by Austrian Science Fund (FWF) Project P25515-N25.
\bibliographystyle{amsplain}

\end{document}